\documentclass{amsart}
\usepackage{amsfonts}

\newtheorem{theorem}{Theorem}
\theoremstyle{plain}

\newtheorem{corollary}{Corollary}

\newtheorem{definition}{Definition}
\newtheorem{example}{Example}

\newtheorem{lemma}{Lemma}

\newtheorem{proposition}{Proposition}
\newtheorem{remark}{Remark}

\numberwithin{equation}{section}
\input{tcilatex}

\begin{document}
\title[Harmonically $(\alpha ,m)$-Convex Functions]{Hermite-Hadamard type
inequalities for harmonically $(\alpha ,m)$-convex functions}
\author{\.{I}mdat \.{I}\c{s}can}
\address{Department of Mathematics, Faculty of Arts and Sciences,\\
Giresun University, 28100, Giresun, Turkey.}
\email{imdati@yahoo.com, imdat.iscan@giresun.edu.tr}

\begin{abstract}
The author introduces the concept of harmonically $(\alpha ,m)$-convex
functions and establishes some Hermite-Hadamard type inequalities of these
classes of functions.
\end{abstract}

\maketitle

\section{Introduction}

Let $f:I\subseteq \mathbb{R\rightarrow R}$ be a convex function defined on
the interval $I$ of real numbers and $a,b\in I$ with $a<b$. The following
double inequality is well known in the literature as Hermite-Hadamard
integral inequality 
\begin{equation}
f\left( \frac{a+b}{2}\right) \leq \frac{1}{b-a}\int\limits_{a}^{b}f(x)dx\leq 
\frac{f(a)+f(b)}{2}\text{.}  \label{1-1}
\end{equation}

The class of $(\alpha ,m)$-convex functions was first introduced In \cite%
{M93}, and it is defined as follows:

\begin{definition}
The function $f:\left[ 0,b\right] \mathbb{\rightarrow R}$, $b>0$, is said to
be $(\alpha ,m)$-convex where $(\alpha ,m)\in \left[ 0,1\right] ^{2}$, if we
have%
\begin{equation*}
f\left( tx+m(1-t)y\right) \leq t^{\alpha }f(x)+m(1-t^{\alpha })f(y)
\end{equation*}%
for all $x,y\in \left[ 0,b\right] $ and $t\in \left[ 0,1\right] $.
\end{definition}

It can be easily that for $(\alpha ,m)\in \left\{ (0,0),(\alpha
,0),(1,0),(1,m),(1,1),(\alpha ,1)\right\} $ one obtains the following
classes of functions: increasing, $\alpha $-starshaped, starshaped, $m$%
-convex, convex, $\alpha $-convex.

Denote by $K_{m}^{\alpha }(b)$ the set of all $(\alpha ,m)$-convex functions
on $\left[ 0,b\right] $ for which $f(0)\leq 0$. For recent results and
generalizations concerning $(\alpha ,m)$-convex functions (see \cite%
{BOP08,I13,I13b,I13c,M93,OAK11,OKS10,OSS11,SSOR09} ).

In \cite{I13d}, the author gave definition of harmonically convex functions
and established some Hermite-Hadamard type inequalities for harmonically
convex functions as follows:

\begin{definition}
Let $I\subset 
\mathbb{R}
\backslash \left\{ 0\right\} $ be a real interval. A function $%
f:I\rightarrow 
\mathbb{R}
$ is said to be harmonically convex, if \ 
\begin{equation}
f\left( \frac{xy}{tx+(1-t)y}\right) \leq tf(y)+(1-t)f(x)  \label{1-3}
\end{equation}%
for all $x,y\in I$ and $t\in \lbrack 0,1]$. If the inequality in (\ref{1-3})
is reversed, then $f$ is said to be harmonically concave.
\end{definition}

\begin{theorem}
Let $f:I\subset 
\mathbb{R}
\backslash \left\{ 0\right\} \rightarrow 
\mathbb{R}
$ be a harmonically convex function and $a,b\in I$ with $a<b.$ If $f\in
L[a,b]$ then the following inequalities hold 
\begin{equation}
f\left( \frac{2ab}{a+b}\right) \leq \frac{ab}{b-a}\int\limits_{a}^{b}\frac{%
f(x)}{x^{2}}dx\leq \frac{f(a)+f(b)}{2}.  \label{1-4}
\end{equation}%
The \ above inequalities are sharp.
\end{theorem}

\begin{theorem}
\label{1.5}Let $f:I\subset \left( 0,\infty \right) \rightarrow 
\mathbb{R}
$ be a differentiable function on $I^{\circ }$, $a,b\in I$ with $a<b,$ and $%
f^{\prime }\in L[a,b].$ If $\left\vert f^{\prime }\right\vert ^{q}$ is
harmonically convex on $[a,b]$ for $q\geq 1,$ then%
\begin{eqnarray}
&&\left\vert \frac{f(a)+f(b)}{2}-\frac{ab}{b-a}\int\limits_{a}^{b}\frac{f(x)%
}{x^{2}}dx\right\vert  \label{1-5} \\
&\leq &\frac{ab\left( b-a\right) }{2}\lambda _{1}^{1-\frac{1}{q}}\left[
\lambda _{2}\left\vert f^{\prime }\left( a\right) \right\vert ^{q}+\lambda
_{3}\left\vert f^{\prime }\left( b\right) \right\vert ^{q}\right] ^{\frac{1}{%
q}},  \notag
\end{eqnarray}%
where 
\begin{eqnarray*}
\lambda _{1} &=&\frac{1}{ab}-\frac{2}{\left( b-a\right) ^{2}}\ln \left( 
\frac{\left( a+b\right) ^{2}}{4ab}\right) , \\
\lambda _{2} &=&\frac{-1}{b\left( b-a\right) }+\frac{3a+b}{\left( b-a\right)
^{3}}\ln \left( \frac{\left( a+b\right) ^{2}}{4ab}\right) , \\
\lambda _{3} &=&\frac{1}{a\left( b-a\right) }-\frac{3b+a}{\left( b-a\right)
^{3}}\ln \left( \frac{\left( a+b\right) ^{2}}{4ab}\right) \\
&=&\lambda _{1}-\lambda _{2}.
\end{eqnarray*}
\end{theorem}

\begin{theorem}
\label{1.6}Let $f:I\subset \left( 0,\infty \right) \rightarrow 
\mathbb{R}
$ be a differentiable function on $I^{\circ }$, $a,b\in I$ with $a<b,$ and $%
f^{\prime }\in L[a,b].$ If $\left\vert f^{\prime }\right\vert ^{q}$ is
harmonically convex on $[a,b]$ for $q>1,\;\frac{1}{p}+\frac{1}{q}=1,$ then%
\begin{eqnarray}
&&\left\vert \frac{f(a)+f(b)}{2}-\frac{ab}{b-a}\int\limits_{a}^{b}\frac{f(x)%
}{x^{2}}dx\right\vert  \label{1-6} \\
&\leq &\frac{ab\left( b-a\right) }{2}\left( \frac{1}{p+1}\right) ^{\frac{1}{p%
}}\left( \mu _{1}\left\vert f^{\prime }\left( a\right) \right\vert ^{q}+\mu
_{2}\left\vert f^{\prime }\left( b\right) \right\vert ^{q}\right) ^{\frac{1}{%
q}},  \notag
\end{eqnarray}%
where%
\begin{eqnarray*}
\mu _{1} &=&\frac{\left[ a^{2-2q}+b^{1-2q}\left[ \left( b-a\right) \left(
1-2q\right) -a\right] \right] }{2\left( b-a\right) ^{2}\left( 1-q\right)
\left( 1-2q\right) }, \\
\mu _{2} &=&\frac{\left[ b^{2-2q}-a^{1-2q}\left[ \left( b-a\right) \left(
1-2q\right) +b\right] \right] }{2\left( b-a\right) ^{2}\left( 1-q\right)
\left( 1-2q\right) }.
\end{eqnarray*}
\end{theorem}

In \cite{I13d}, the author gave the following identity for differentiable
functions.

\begin{lemma}
\label{1.1}Let $f:I\subset 
\mathbb{R}
\backslash \left\{ 0\right\} \rightarrow 
\mathbb{R}
$ be a differentiable function on $I^{\circ }$ and $a,b\in I$ with $a<b$. If 
$f^{\prime }\in L[a,b]$ then 
\begin{eqnarray*}
&&\frac{f(a)+f(b)}{2}-\frac{ab}{b-a}\int\limits_{a}^{b}\frac{f(x)}{x^{2}}dx
\\
&=&\frac{ab\left( b-a\right) }{2}\int\limits_{0}^{1}\frac{1-2t}{\left(
tb+(1-t)a\right) ^{2}}f^{\prime }\left( \frac{ab}{tb+(1-t)a}\right) dt.
\end{eqnarray*}
\end{lemma}

The main purpose of this paper is to introduce the concept of harmonically $%
(\alpha ,m)$-convex functions and establish some new Hermite-Hadamard type
inequalities for these classes of functions.

\section{Main Results}

\begin{definition}
\label{2.1}The function $f:\left( 0,b^{\ast }\right] \rightarrow 
\mathbb{R}
,$ $b^{\ast }>0,$ is said to be harmonically $(\alpha ,m)$-convex, where $%
\alpha \in \left[ 0,1\right] $ and $m\in \left( 0,1\right] $, if \ 
\begin{equation}
f\left( \frac{mxy}{mty+(1-t)x}\right) =f\left( \left( \frac{t}{x}+\frac{1-t}{%
my}\right) ^{-1}\right) \leq t^{\alpha }f(x)+m(1-t^{\alpha })f(y)
\label{2-1}
\end{equation}%
for all $x,y\in \left( 0,b^{\ast }\right] $ and $t\in \lbrack 0,1]$. If the
inequality in (\ref{2-1}) is reversed, then $f$ is said to be harmonically $%
(\alpha ,m)$-concave.
\end{definition}

\begin{remark}
When $m=\alpha =1$, the harmonically $(\alpha ,m)$-convex (concave) function
defined in Definition \ref{2.1} becomes a harmonically convex (concave)
function defined in \cite{I13d}. Thus, every harmonically convex (concave)
function is also harmonically $(1,1)$-convex (concave) function.
\end{remark}

The following proposition is obvious.

\begin{proposition}
\label{2.0}Let $f:\left( 0,b^{\ast }\right] \rightarrow 
\mathbb{R}
$ be a function.

\begin{itemize}
\item[a)] if $f$ is $(\alpha ,m)$-convex and nondecreasing function then f
is harmonically $(\alpha ,m)$-convex.

\item[b)] if $f$ is harmonically $(\alpha ,m)$-convex and nonincreasing
function then f is $(\alpha ,m)$-convex.
\end{itemize}
\end{proposition}

\begin{proof}
For all $t\in \left[ 0,1\right] $, $m\in \left( 0,1\right] $ and $x,y\in
\left( 0,b^{\ast }\right] $ we have%
\begin{equation*}
t(1-t)\left( x-my\right) ^{2}\geq 0,
\end{equation*}%
then the following inequality holds%
\begin{equation}
\frac{mxy}{mty+(1-t)x}\leq tx+m(1-t)y.  \label{2-0}
\end{equation}%
By the inequality (\ref{2-0}), the proof is completed.
\end{proof}

\begin{remark}
According to Proposition \ref{2.0}, every nondecreasing $s$-convex function
in the first sense (or $(s,1)$-convex function) is also harmonically $(s,1)$%
-convex function.
\end{remark}

\begin{example}
Let $s\in \left( 0,1\right] $, then the function $f:\left( 0,\infty \right)
\rightarrow 
\mathbb{R}
,\ f(x)=x^{s}$ is a nondecreasing $s$-convex function in the first sense 
\cite{HM94}. According to the above Remark $f$ is also harmonically $(s,1)$%
-convex function.
\end{example}

\begin{proposition}
Let $\alpha \in \left[ 0,1\right] $, $m\in \left( 0,1\right] ,$ $f:\left[
a,mb\right] \subset \left( 0,b^{\ast }\right] \rightarrow 
\mathbb{R}
$, $b^{\ast }>0,$ be a function and $g:\left[ a,mb\right] \rightarrow \left[
a,mb\right] $, $g\left( x\right) =\frac{mab}{a+mb-x}$, $a<mb$. Then $f$ is
harmonically $(\alpha ,m)$-convex on $\left[ a,mb\right] $ if and only if $%
f\circ g$ is $(\alpha ,m)$-convex on $\left[ a,mb\right] .$
\end{proposition}

\begin{proof}
Since 
\begin{equation}
\left( f\circ g\right) (ta+m(1-t)b)=f\left( \frac{mab}{mtb+(1-t)a}\right) ,
\label{p-1}
\end{equation}%
for all $t\in \left[ 0,1\right] $, $m\in \left( 0,1\right] $. The proof is
obvious from equality (\ref{p-1}).
\end{proof}

The following result of the Hermite-Hadamard type holds.

\begin{theorem}
\label{2.2} Let $f:\left( 0,\infty \right) \rightarrow 
\mathbb{R}
$ be a harmonically $(\alpha ,m)$-convex function with $\alpha \in \left[ 0,1%
\right] $ and $m\in \left( 0,1\right] $. If $0<$ $a<b<\infty $ and $f\in
L[a,b]$, then one has the inequality 
\begin{equation}
\frac{ab}{b-a}\int\limits_{a}^{b}\frac{f(x)}{x^{2}}dx\leq \min \left\{ \frac{%
f(a)+\alpha mf(\frac{b}{m})}{\alpha +1},\frac{f(b)+\alpha mf(\frac{a}{m})}{%
\alpha +1}\right\} .  \label{2-2}
\end{equation}
\end{theorem}

\begin{proof}
Since $f:\left( 0,\infty \right) \rightarrow 
\mathbb{R}
$ is a harmonically $(\alpha ,m)$-convex function, we have, for all $x,y\in
I $%
\begin{equation*}
f\left( \frac{xy}{ty+(1-t)x}\right) =f\left( \frac{m\frac{y}{m}x}{tm\frac{y}{%
m}+(1-t)x}\right) \leq t^{\alpha }f(x)+m(1-t^{\alpha })f(\frac{y}{m})
\end{equation*}%
which gives: 
\begin{equation*}
f\left( \frac{ab}{tb+(1-t)a}\right) \leq t^{\alpha }f(a)+m(1-t^{\alpha })f(%
\frac{b}{m})
\end{equation*}%
and%
\begin{equation*}
f\left( \frac{ab}{ta+(1-t)b}\right) \leq t^{\alpha }f(b)+m(1-t^{\alpha })f(%
\frac{a}{m})
\end{equation*}%
for all $t\in \lbrack 0,1].$ Integrating on $[0,1]$ we obtain%
\begin{equation*}
\int\limits_{0}^{1}f\left( \frac{ab}{tb+(1-t)a}\right) dt\leq \frac{%
f(a)+\alpha mf(\frac{b}{m})}{\alpha +1}
\end{equation*}%
and%
\begin{equation*}
\int\limits_{0}^{1}f\left( \frac{ab}{ta+(1-t)b}\right) dt\leq \frac{%
f(b)+\alpha mf(\frac{a}{m})}{\alpha +1}.
\end{equation*}%
However, 
\begin{equation*}
\int\limits_{0}^{1}f\left( \frac{ab}{tb+(1-t)a}\right)
dt=\int\limits_{0}^{1}f\left( \frac{ab}{ta+(1-t)b}\right) dt=\frac{ab}{b-a}%
\int\limits_{a}^{b}\frac{f(x)}{x^{2}}dx
\end{equation*}%
and the inequality (\ref{2-2}) is obtained.
\end{proof}

\begin{remark}
If we take $\alpha =m=1$ in Theorem \ref{2.2}, then inequality (\ref{2-2})
becomes the right-hand side of inequality (\ref{1-4}).
\end{remark}

\begin{corollary}
\label{c.1}If we take $m=1$ in Theorem \ref{2.2}, then we get%
\begin{equation}
\frac{ab}{b-a}\int\limits_{a}^{b}\frac{f(x)}{x^{2}}dx\leq \min \left\{ \frac{%
f(a)+\alpha f(b)}{\alpha +1},\frac{f(b)+\alpha f(a)}{\alpha +1}\right\}
\label{c-1}
\end{equation}
\end{corollary}

\begin{theorem}
\label{2.3}Let $f:I\subset \left( 0,\infty \right) \rightarrow 
\mathbb{R}
$ be a differentiable function on $I^{\circ }$, $a,b/m\in I^{\circ }$ with $%
a<b,$ $m\in \left( 0,1\right] $ and $f^{\prime }\in L[a,b].$ If $\left\vert
f^{\prime }\right\vert ^{q}$ is harmonically $(\alpha ,m)$-convex on $%
[a,b/m] $ for $q\geq 1$, with $\alpha \in \left[ 0,1\right] $, then%
\begin{equation*}
\left\vert \frac{f(a)+f(b)}{2}-\frac{ab}{b-a}\int\limits_{a}^{b}\frac{f(x)}{%
x^{2}}dx\right\vert
\end{equation*}%
\begin{equation*}
\leq \frac{ab\left( b-a\right) }{2^{2-1/q}}\left[ \lambda (\alpha
,q;a,b)\left\vert f^{\prime }\left( a\right) \right\vert ^{q}+m\mu (\alpha
,q;a,b)\left\vert f^{\prime }\left( b/m\right) \right\vert ^{q}\right] ^{%
\frac{1}{q}},
\end{equation*}%
where 
\begin{eqnarray*}
\lambda (\alpha ,q;a,b) &=&\frac{\beta \left( 1,\alpha +2\right) }{b^{2q}}%
._{2}F_{1}\left( 2q,1;\alpha +3;1-\frac{a}{b}\right) -\frac{\beta \left(
2,\alpha +1\right) }{b^{2q}}._{2}F_{1}\left( 2q,2;\alpha +3;1-\frac{a}{b}%
\right) \\
&&+\frac{2^{2q-\alpha }\beta \left( 2,\alpha +1\right) }{\left( a+b\right)
^{2q}}._{2}F_{1}\left( 2q,2;\alpha +3;1-\frac{2a}{a+b}\right) , \\
\mu (\alpha ,q;a,b) &=&\lambda (0,q;a,b)-\lambda (\alpha ,q;a,b),
\end{eqnarray*}%
$\beta $ is Euler Beta function defined by%
\begin{equation*}
\beta \left( x,y\right) =\frac{\Gamma (x)\Gamma (y)}{\Gamma (x+y)}%
=\int\limits_{0}^{1}t^{x-1}\left( 1-t\right) ^{y-1}dt,\ \ x,y>0,
\end{equation*}%
and $_{2}F_{1}$ is hypergeometric function defined by 
\begin{equation*}
_{2}F_{1}\left( a,b;c;z\right) =\frac{1}{\beta \left( b,c-b\right) }%
\int\limits_{0}^{1}t^{b-1}\left( 1-t\right) ^{c-b-1}\left( 1-zt\right)
^{-a}dt,\ c>b>0,\ \left\vert z\right\vert <1\text{ (see \cite{AS65}).}
\end{equation*}
\end{theorem}

\begin{proof}
From Lemma \ref{1.1} and using the power mean inequality, we have%
\begin{eqnarray*}
&&\left\vert \frac{f(a)+f(b)}{2}-\frac{ab}{b-a}\int\limits_{a}^{b}\frac{f(x)%
}{x^{2}}dx\right\vert \\
&\leq &\frac{ab\left( b-a\right) }{2}\int\limits_{0}^{1}\left\vert \frac{1-2t%
}{\left( tb+(1-t)a\right) ^{2}}\right\vert \left\vert f^{\prime }\left( 
\frac{ab}{tb+(1-t)a}\right) \right\vert dt \\
&\leq &\frac{ab\left( b-a\right) }{2}\left( \int\limits_{0}^{1}\left\vert
1-2t\right\vert dt\right) ^{1-\frac{1}{q}} \\
&&\times \left( \int\limits_{0}^{1}\frac{\left\vert 1-2t\right\vert }{\left(
tb+(1-t)a\right) ^{2q}}\left\vert f^{\prime }\left( \frac{ab}{tb+(1-t)a}%
\right) \right\vert ^{q}dt\right) ^{\frac{1}{q}}.
\end{eqnarray*}%
Hence, by harmonically $(\alpha ,m)$-convexity of $\left\vert f^{\prime
}\right\vert ^{q}$ on $[a,b/m],$we have 
\begin{eqnarray*}
&&\left\vert \frac{f(a)+f(b)}{2}-\frac{ab}{b-a}\int\limits_{a}^{b}\frac{f(x)%
}{x^{2}}dx\right\vert \\
&\leq &\frac{ab\left( b-a\right) }{2}\left( \frac{1}{2}\right) ^{1-\frac{1}{q%
}}\left( \int\limits_{0}^{1}\frac{\left\vert 1-2t\right\vert \left[
t^{\alpha }\left\vert f^{\prime }\left( a\right) \right\vert
^{q}+m(1-t^{\alpha })\left\vert f^{\prime }\left( b/m\right) \right\vert ^{q}%
\right] }{\left( tb+(1-t)a\right) ^{2q}}dt\right) ^{\frac{1}{q}} \\
&\leq &\frac{ab\left( b-a\right) }{2^{2-1/q}}\left[ \lambda (\alpha
,q;a,b)\left\vert f^{\prime }\left( a\right) \right\vert ^{q}+m\left(
\lambda (0,q;a,b)-\lambda (\alpha ,q;a,b)\right) \left\vert f^{\prime
}\left( b/m\right) \right\vert ^{q}\right] ^{\frac{1}{q}}.
\end{eqnarray*}%
It is easily check that 
\begin{eqnarray*}
&&\int\limits_{0}^{1}\frac{\left\vert 1-2t\right\vert t^{\alpha }}{\left(
tb+(1-t)a\right) ^{2q}}dt=2\int\limits_{0}^{1/2}\frac{\left( 1-2t\right)
t^{\alpha }}{\left( tb+(1-t)a\right) ^{2q}}dt-\int\limits_{0}^{1}\frac{%
\left( 1-2t\right) t^{\alpha }}{\left( tb+(1-t)a\right) ^{2q}}dt \\
&=&\frac{\beta \left( 1,\alpha +2\right) }{b^{2q}}._{2}F_{1}\left(
2q,1;\alpha +3;1-\frac{a}{b}\right) -\frac{\beta \left( 2,\alpha +1\right) }{%
b^{2q}}._{2}F_{1}\left( 2q,2;\alpha +3;1-\frac{a}{b}\right) \\
&&+\frac{2^{2q-\alpha }\beta \left( 2,\alpha +1\right) }{\left( a+b\right)
^{2q}}._{2}F_{1}\left( 2q,2;\alpha +3;1-\frac{2a}{a+b}\right) =\lambda
(\alpha ,q;a,b)
\end{eqnarray*}%
\begin{equation*}
\int\limits_{0}^{1}\frac{\left\vert 1-2t\right\vert \left( 1-t^{\alpha
}\right) }{\left( tb+(1-t)a\right) ^{2q}}dt=\lambda (0,q;a,b)-\lambda
(\alpha ,q;a,b).
\end{equation*}%
This completes the proof.
\end{proof}

If we take $\alpha =m=1$ in Theorem \ref{2.3} then we get the following a
new corrollary for harmonically convex functions:

\begin{corollary}
Let $f:I\subset \left( 0,\infty \right) \rightarrow 
\mathbb{R}
$ be a differentiable function on $I^{\circ }$, $a,b\in I^{\circ }$ with $%
a<b $ and $f^{\prime }\in L[a,b].$ If $\left\vert f^{\prime }\right\vert
^{q} $ is harmonically convex on $[a,b]$ for $q\geq 1$ then%
\begin{equation*}
\left\vert \frac{f(a)+f(b)}{2}-\frac{ab}{b-a}\int\limits_{a}^{b}\frac{f(x)}{%
x^{2}}dx\right\vert
\end{equation*}%
\begin{equation*}
\leq \frac{ab\left( b-a\right) }{2^{2-1/q}}\left[ \lambda
(1,q;a,b)\left\vert f^{\prime }\left( a\right) \right\vert ^{q}+\mu
(1,q;a,b)\left\vert f^{\prime }\left( b/m\right) \right\vert ^{q}\right] ^{%
\frac{1}{q}}.
\end{equation*}
\end{corollary}

\begin{corollary}
\label{c.2}If we take $m=1$ in Theorem \ref{2.3} then we get 
\begin{equation}
\left\vert \frac{f(a)+f(b)}{2}-\frac{ab}{b-a}\int\limits_{a}^{b}\frac{f(x)}{%
x^{2}}dx\right\vert  \label{c-2}
\end{equation}%
\begin{equation*}
\leq \frac{ab\left( b-a\right) }{2^{2-1/q}}\left[ \lambda (\alpha
,q;a,b)\left\vert f^{\prime }\left( a\right) \right\vert ^{q}+\mu (\alpha
,q;a,b)\left\vert f^{\prime }\left( b\right) \right\vert ^{q}\right] ^{\frac{%
1}{q}}.
\end{equation*}
\end{corollary}

\begin{theorem}
\label{2.4}Let $f:I\subset \left( 0,\infty \right) \rightarrow 
\mathbb{R}
$ be a differentiable function on $I^{\circ }$, $a,b/m\in I^{\circ }$ with $%
a<b,$ $m\in \left( 0,1\right] $ and $f^{\prime }\in L[a,b].$ If $\left\vert
f^{\prime }\right\vert ^{q}$ is harmonically $(\alpha ,m)$-convex on $%
[a,b/m] $ for $q\geq 1$, with $\alpha \in \left[ 0,1\right] ,$ then%
\begin{equation}
\left\vert \frac{f(a)+f(b)}{2}-\frac{ab}{b-a}\int\limits_{a}^{b}\frac{f(x)}{%
x^{2}}dx\right\vert  \label{2-4}
\end{equation}%
\begin{equation*}
\leq \frac{ab\left( b-a\right) }{2}\lambda ^{1-\frac{1}{q}}(0,q;a,b)\left[
\lambda (\alpha ,1;a,b)\left\vert f^{\prime }\left( a\right) \right\vert
^{q}+m\mu (\alpha ,1;a,b)\left\vert f^{\prime }\left( b/m\right) \right\vert
^{q}\right] ^{\frac{1}{q}},
\end{equation*}%
where $\lambda $ and $\mu $ is defined as in Theorem \ref{2.3}.
\end{theorem}

\begin{proof}
From Lemma \ref{1.1}, power mean inequality and the harmonically $(\alpha
,m) $-convexity of $\left\vert f^{\prime }\right\vert ^{q}$ on $[a,b/m],$we
have,%
\begin{eqnarray*}
&&\left\vert \frac{f(a)+f(b)}{2}-\frac{ab}{b-a}\int\limits_{a}^{b}\frac{f(x)%
}{x^{2}}dx\right\vert \\
&\leq &\frac{ab\left( b-a\right) }{2}\int\limits_{0}^{1}\left\vert \frac{1-2t%
}{\left( tb+(1-t)a\right) ^{2}}\right\vert \left\vert f^{\prime }\left( 
\frac{ab}{tb+(1-t)a}\right) \right\vert dt \\
&\leq &\frac{ab\left( b-a\right) }{2}\left( \int\limits_{0}^{1}\left\vert 
\frac{1-2t}{\left( tb+(1-t)a\right) ^{2}}\right\vert dt\right) ^{1-\frac{1}{q%
}} \\
&&\times \left( \int\limits_{0}^{1}\frac{\left\vert 1-2t\right\vert \left[
t^{\alpha }\left\vert f^{\prime }\left( a\right) \right\vert
^{q}+m(1-t^{\alpha })\left\vert f^{\prime }\left( b/m\right) \right\vert ^{q}%
\right] }{\left( tb+(1-t)a\right) ^{2}}dt\right) ^{\frac{1}{q}} \\
&\leq &\frac{ab\left( b-a\right) }{2}\lambda ^{1-\frac{1}{q}}(0,q;a,b)\left[
\lambda (\alpha ,1;a,b)\left\vert f^{\prime }\left( a\right) \right\vert
^{q}+m\mu (\alpha ,1;a,b)\left\vert f^{\prime }\left( b/m\right) \right\vert
^{q}\right] ^{\frac{1}{q}}.
\end{eqnarray*}
\end{proof}

\begin{remark}
If we take $\alpha =m=1$ in Theorem \ref{2.4} then inequality (\ref{2-4})
becomes inequality (\ref{1-5}) of Theorem \ref{1.5}.
\end{remark}

\begin{corollary}
\label{c.3}If we take $m=1$ in Theorem \ref{2.4} then we get%
\begin{equation}
\left\vert \frac{f(a)+f(b)}{2}-\frac{ab}{b-a}\int\limits_{a}^{b}\frac{f(x)}{%
x^{2}}dx\right\vert  \label{c-3}
\end{equation}%
\begin{equation*}
\leq \frac{ab\left( b-a\right) }{2}\lambda ^{1-\frac{1}{q}}(0,q;a,b)\left[
\lambda (\alpha ,1;a,b)\left\vert f^{\prime }\left( a\right) \right\vert
^{q}+\mu (\alpha ,1;a,b)\left\vert f^{\prime }\left( b\right) \right\vert
^{q}\right] ^{\frac{1}{q}},
\end{equation*}
\end{corollary}

\begin{theorem}
\label{2.5}Let $f:I\subset \left( 0,\infty \right) \rightarrow 
\mathbb{R}
$ be a differentiable function on $I^{\circ }$, $a,b/m\in I^{\circ }$ with $%
a<b,$ $m\in \left( 0,1\right] ,$ and $f^{\prime }\in L[a,b].$ If $\left\vert
f^{\prime }\right\vert ^{q}$ is harmonically $(\alpha ,m)$-convex on $%
[a,b/m] $ for $q>1,\;\frac{1}{p}+\frac{1}{q}=1,$ with $\alpha \in \left[ 0,1%
\right] , $ then%
\begin{equation}
\left\vert \frac{f(a)+f(b)}{2}-\frac{ab}{b-a}\int\limits_{a}^{b}\frac{f(x)}{%
x^{2}}dx\right\vert  \label{2-5}
\end{equation}%
\begin{equation*}
\leq \frac{ab\left( b-a\right) }{2}\left( \frac{1}{p+1}\right) ^{\frac{1}{p}%
}\left( \nu (\alpha ,q;a,b)\left\vert f^{\prime }\left( a\right) \right\vert
^{q}+m(\nu (0,q;a,b)-\nu (\alpha ,q;a,b))\left\vert f^{\prime }\left(
b/m\right) \right\vert ^{q}\right) ^{\frac{1}{q}}
\end{equation*}%
where%
\begin{equation*}
\nu (\alpha ,q;a,b)=\frac{\beta \left( 1,\alpha +1\right) }{b^{2q}}%
._{2}F_{1}\left( 2q,1;\alpha +2;1-\frac{a}{b}\right) .
\end{equation*}
\end{theorem}

\begin{proof}
From Lemma \ref{1.1}, H\"{o}lder's inequality and the harmonically $(\alpha
,m)$-convexity of $\left\vert f^{\prime }\right\vert ^{q}$ on $[a,b/m],$we
have, 
\begin{eqnarray*}
&&\left\vert \frac{f(a)+f(b)}{2}-\frac{ab}{b-a}\int\limits_{a}^{b}\frac{f(x)%
}{x^{2}}dx\right\vert \\
&\leq &\frac{ab\left( b-a\right) }{2}\left( \int\limits_{0}^{1}\left\vert
1-2t\right\vert ^{p}dt\right) ^{\frac{1}{p}} \\
&&\times \left( \int\limits_{0}^{1}\frac{1}{\left( tb+(1-t)a\right) ^{2q}}%
\left\vert f^{\prime }\left( \frac{ab}{tb+(1-t)a}\right) \right\vert
^{q}dt\right) ^{\frac{1}{q}}
\end{eqnarray*}%
\begin{eqnarray*}
&\leq &\frac{ab\left( b-a\right) }{2}\left( \frac{1}{p+1}\right) ^{\frac{1}{p%
}}\left( \int\limits_{0}^{1}\frac{t^{\alpha }\left\vert f^{\prime }\left(
a\right) \right\vert ^{q}+m(1-t^{\alpha })\left\vert f^{\prime }\left(
b/m\right) \right\vert ^{q}}{\left( tb+(1-t)a\right) ^{2q}}dt\right) ^{\frac{%
1}{q}} \\
&\leq &\frac{ab\left( b-a\right) }{2}\left( \frac{1}{p+1}\right) ^{\frac{1}{p%
}}\left( \nu (\alpha ,q;a,b)\left\vert f^{\prime }\left( a\right)
\right\vert ^{q}+m(\nu (0,q;a,b)-\nu (\alpha ,q;a,b))\left\vert f^{\prime
}\left( b/m\right) \right\vert ^{q}\right) ^{\frac{1}{q}},
\end{eqnarray*}%
where an easy calculation gives%
\begin{eqnarray*}
&&\int\limits_{0}^{1}\frac{t^{\alpha }}{\left( tb+(1-t)a\right) ^{2q}}dt \\
&=&\frac{\beta \left( 1,\alpha +1\right) }{b^{2q}}._{2}F_{1}\left(
2q,1;\alpha +2;1-\frac{a}{b}\right) \\
&=&\nu (\alpha ,q;a,b)
\end{eqnarray*}%
and%
\begin{eqnarray*}
&&\int\limits_{0}^{1}\frac{1-t^{\alpha }}{\left( tb+(1-t)a\right) ^{2q}}dt \\
&=&\nu (0,q;a,b)-\nu (\alpha ,q;a,b).
\end{eqnarray*}%
This completes the proof.
\end{proof}

\begin{remark}
If we take $\alpha =m=1$ in Theorem \ref{2.5} then inequality (\ref{2-5})
becomes inequality (\ref{1-6}) of Theorem \ref{1.6}.
\end{remark}

\begin{corollary}
\label{c.4}If we take $m=1$ in Theorem \ref{2.5} then we get%
\begin{equation}
\left\vert \frac{f(a)+f(b)}{2}-\frac{ab}{b-a}\int\limits_{a}^{b}\frac{f(x)}{%
x^{2}}dx\right\vert  \label{c-4}
\end{equation}
\begin{equation*}
\leq \frac{ab\left( b-a\right) }{2}\left( \frac{1}{p+1}\right) ^{\frac{1}{p}%
}\left( \nu (\alpha ,q;a,b)\left\vert f^{\prime }\left( a\right) \right\vert
^{q}+(\nu (0,q;a,b)-\nu (\alpha ,q;a,b))\left\vert f^{\prime }\left(
b\right) \right\vert ^{q}\right) ^{\frac{1}{q}}.
\end{equation*}
\end{corollary}

\section{Some applications for special means}

Let us recall the following special means of two nonnegative number $a,b$
with $b>a:$

\begin{enumerate}
\item The weighted arithmetic mean%
\begin{equation*}
A_{\alpha }\left( a,b\right) :=\alpha a+(1-\alpha )b,~\alpha \in \left[ 0,1%
\right] .
\end{equation*}

\item The arithmetic mean%
\begin{equation*}
A=A\left( a,b\right) :=\frac{a+b}{2}.
\end{equation*}

\item The geometric mean%
\begin{equation*}
G=G\left( a,b\right) :=\sqrt{ab}.
\end{equation*}

\item The harmonic mean%
\begin{equation*}
H=H\left( a,b\right) :=\frac{2ab}{a+b}.
\end{equation*}

\item The p-Logarithmic mean%
\begin{equation*}
L_{p}=L_{p}\left( a,b\right) :=\left( \frac{b^{p+1}-a^{p+1}}{(p+1)(b-a)}%
\right) ^{\frac{1}{p}},\ \ p\in 
\mathbb{R}
\backslash \left\{ -1,0\right\} .
\end{equation*}
\end{enumerate}

\begin{proposition}
Let $0<a<b.$ Then we have the following inequality%
\begin{equation*}
G^{2}L_{\alpha -2}^{\alpha -2}\leq \min \left\{ A_{1/(\alpha +1)}\left(
a^{\alpha },b^{\alpha }\right) ,A_{1/(\alpha +1)}\left( b^{\alpha
},a^{\alpha }\right) \right\} .
\end{equation*}
\end{proposition}

\begin{proof}
The assertion follows from the inequality (\ref{c-1}) in Corollary \ref{c.1}%
, for $f:\left( 0,\infty \right) \rightarrow 
\mathbb{R}
,\ f(x)=x^{\alpha },\ 0<\alpha <1.$
\end{proof}

\begin{proposition}
Let $0<a<b$, $q\geq 1$ and $0<\alpha <1.$ Then we have the following
inequality%
\begin{equation*}
\left\vert A\left( a^{\frac{\alpha }{q}+1},b^{\frac{\alpha }{q}+1}\right)
-G^{2}L_{\frac{\alpha }{q}-1}^{\frac{\alpha }{q}-1}\right\vert
\end{equation*}%
\begin{equation*}
\leq \frac{ab\left( b-a\right) (\alpha +q)}{q2^{2-1/q}}\left[ \lambda
(\alpha ,q;a,b)a^{\alpha }+\mu (\alpha ,q;a,b)b^{\alpha }\right] ^{\frac{1}{q%
}}.
\end{equation*}
\end{proposition}

\begin{proof}
The assertion follows from the inequality (\ref{c-2}) in Corollary \ref{c.2}%
, for $f:\left( 0,\infty \right) \rightarrow 
\mathbb{R}
,\ f(x)=x^{\frac{\alpha }{q}+1}/\left( \frac{\alpha }{q}+1\right) .$
\end{proof}

\begin{proposition}
Let $0<a<b$, $q\geq 1$ and $0<\alpha <1.$ Then we have the following
inequality%
\begin{equation*}
\left\vert A\left( a^{\frac{\alpha }{q}+1},b^{\frac{\alpha }{q}+1}\right)
-G^{2}L_{\frac{\alpha }{q}-1}^{\frac{\alpha }{q}-1}\right\vert
\end{equation*}%
\begin{equation*}
\leq \frac{ab\left( b-a\right) (\alpha +q)}{2q}\lambda ^{1-\frac{1}{q}%
}(0,q;a,b)\left[ \lambda (\alpha ,1;a,b)a^{\alpha }+\mu (\alpha
,1;a,b)b^{\alpha }\right] ^{\frac{1}{q}},
\end{equation*}
\end{proposition}

\begin{proof}
The assertion follows from the inequality (\ref{c-3}) in Corollary \ref{c.3}%
, for $f:\left( 0,\infty \right) \rightarrow 
\mathbb{R}
,\ f(x)=x^{\frac{\alpha }{q}+1}/\left( \frac{\alpha }{q}+1\right) .$
\end{proof}

\begin{proposition}
Let $0<a<b$, $q>1$, $1/p+1/q=1$ and $0<\alpha <1.$ Then we have the
following inequality%
\begin{equation*}
\left\vert A\left( a^{\frac{\alpha }{q}+1},b^{\frac{\alpha }{q}+1}\right)
-G^{2}L_{\frac{\alpha }{q}-1}^{\frac{\alpha }{q}-1}\right\vert
\end{equation*}%
\begin{equation*}
\leq \frac{ab\left( b-a\right) (\alpha +q)}{2q}\left( \frac{1}{p+1}\right) ^{%
\frac{1}{p}}\left( \nu (\alpha ,q;a,b)a^{\alpha }+(\nu (0,q;a,b)-\nu (\alpha
,q;a,b))b^{\alpha }\right) ^{\frac{1}{q}}.
\end{equation*}
\end{proposition}

\begin{proof}
The assertion follows from the inequality (\ref{c-4}) in Corollary \ref{c.4}%
, for $f:\left( 0,\infty \right) \rightarrow 
\mathbb{R}
,\ f(x)=x^{\frac{\alpha }{q}+1}/\left( \frac{\alpha }{q}+1\right) .$
\end{proof}

\end{document}